\documentclass[12 pt,twoside]{amsart}

\usepackage{amsopn}
\usepackage{amssymb}
\usepackage{amscd}

\newtheorem{theorem}{Theorem}[section]
\newtheorem{lemma}[theorem]{Lemma}
\newtheorem{proposition}[theorem]{Proposition}
\newtheorem{corollary}[theorem]{Corollary}

\theoremstyle{definition}
\newtheorem{definition}[theorem]{Definition}

\theoremstyle{remark}

\numberwithin{equation}{section}

\begin{document}

\title[Shift perturbations of the identity]{Dynamics of perturbations of the identity operator by multiples of the backward shift
on $l^{\infty}(\mathbb{N})$}

\author[G. Costakis]{George Costakis}
\address{Department of Mathematics, University of Crete, Knossos Avenue, GR-714 09 Heraklion, Crete, Greece}
\email{costakis@math.uoc.gr}
\thanks{}

\author[A. Manoussos]{Antonios Manoussos}
\address{Fakult\"{a}t f\"{u}r Mathematik, SFB 701, Universit\"{a}t Bielefeld, Postfach 100131, D-33501 Bielefeld, Germany}
\email{amanouss@math.uni-bielefeld.de} \urladdr{http://www.math.uni-bielefeld.de/$\sim$amanouss}
\thanks{During this research the second author was fully supported by SFB 701 ``Spektrale Strukturen und
Topologische Methoden in der Mathematik" at the University of Bielefeld, Germany.}

\author[A.B. Nasseri]{Amir Bahman Nasseri}
\address{Fachbereich C - Mathematik und Naturwissenschaften Arbeits\-grup\-pe Funktionalanalysis, Bergische Universit\"{a}t Wuppertal, D- 42097 Wuppertal, Germany}
\email{nasseri@math.uni-wuppertal.de}
\thanks{}

\subjclass[2010]{47A16}

\date{}

\keywords{backward shift, locally topologically transitive operator, $J$-class operator, locally topologically
mixing operator, $J^{mix}$-class operator}

\begin{abstract}
Let $B$, $I$ be the unweighted backward shift and the identity operator respectively on
$l^{\infty}(\mathbb{N})$, the space of bounded sequences over the complex numbers endowed with the supremum
norm. We prove that $I+\lambda B$ is locally topologically transitive if and only if $|\lambda |>2$. This, shows
that a classical result of Salas, which says that backward shift perturbations of the identity operator are
always hypercyclic, or equivalently topologically transitive, on $l^p(\mathbb{N})$, $1\leq p<+\infty$, fails to
hold for the notion of local topological transitivity on $l^{\infty}(\mathbb{N})$. We also obtain further
results which complement certain results from \cite{CosMa}.
\end{abstract}

\maketitle

\section{Introduction}
We start with some terminology and notation. Throughout this paper the letter $X$ stands for a Banach space, $T:X\to X$ will always be a bounded linear operator and
frequently we shall drop the words ``bounded linear" for the sake of simplicity. The symbol $L(X)$ stands for the set of all bounded linear operators acting on $X$. As
usual, by $\mathbb{N}$, $\mathbb{R}$, $\mathbb{C}$ we denote the sets of positive integers, real numbers and complex numbers respectively. In the following, the symbols
$\mathbb{D}$, $\overline{\mathbb{D}}$, $\partial \mathbb{D}$ denote the open unit disk, closed unit disk and the unit circle on the complex plane respectively.\smallskip

The purpose of this note is to explore the ``local" dynamics of certain operators on $l^{\infty}(\mathbb{N})$, where $l^{\infty}(\mathbb{N})$ denotes the Banach space of
bounded sequences of complex numbers endowed with the usual supremum norm. We will frequently use the subspace $c_0(\mathbb{N})$ of $l^{\infty}(\mathbb{N})$, which
consists of all the null sequences of complex numbers. Let us proceed with the relevant definitions.

\begin{definition}
Let $T:X\rightarrow X$ be an operator. For every $x\in X$ the sets

\[
\begin{split}
J_T(x)=\{ &y\in X: \mbox{ there exist a strictly increasing sequence of positive}\\
&\mbox{integers}\,(k_{n})\,\mbox{and a sequence }\, (x_{n})\subset X\,\mbox{such that}\, x_{n}\rightarrow x\,
\mbox{and}\\
&T^{k_{n}}x_{n}\rightarrow y\}
\end{split}
\]
\[
\begin{split}
J_T^{mix}(x)=\{ &y\in X:\, \mbox{there exists a sequence }\,(x_{n})\subset X\,\mbox{such that}\\
&\, x_{n}\rightarrow x\, \mbox{and} \,\,T^nx_{n}\rightarrow y\}
\end{split}
\]
denote the extended (prolongational) limit set and the extended mixing limit set of $x$ under $T$ respectively.
\end{definition}

\begin{definition}
Let $X$ be a Banach space and $T:X\rightarrow X$ be an operator. Then $T$ is called \textit{topologically
transitive} if for every pair of open sets $U,V$ of $X$ there exists a positive integer $n$ such that $T^nU\cap
V\neq \emptyset $.  The operator $T$ is called \textit{topologically mixing} if for every pair of open sets
$U,V$ of $X$ there exists a positive integer $k$ such that $T^nU\cap V\neq \emptyset $ for every $n\geq k$.
\end{definition}

\begin{definition}
Let $X$ be a separable Banach space. An operator $T:X\to X$ is called \textit{hypercyclic} if there exists a
vector $x\in X$ so that its orbit under $T$, $Orb(T,x):=\{ x, Tx, T^2, \ldots \}$, is dense in $X$. Then $x$ is
called \textit{hypercyclic} for $T$.
\end{definition}

It is easy to see that if $X$ is separable, then an operator $T$ on $X$ is hypercyclic if and only if $T$ is
topologically transitive. Clearly the notion of hypercyclicity makes no sense when $X$ is non-separable.
However, one can find topologically transitive operators on non-separable Banach spaces, see for instance
\cite{MFPW}. On the other hand not every non-separable Banach space supports topologically transitive operators,
as for example $l^{\infty}(\mathbb{N})$ \cite{BeKa}. For a comprehensive treatment on hypercyclicity we refer to
the recent books \cite{BaMa}, \cite{GEPe}. \smallskip

We now move to ``localized analogues" of hypercyclicity and topological transitivity, introduced in \cite{cosma1}.
\begin{definition}
An operator $T:X\rightarrow X $ will be called \textit{$J$-class or locally topologically transitive} (\textit{$J^{mix}$-class or locally topologically mixing}) provided
there exists a non-zero vector $x\in X$ so that the extended limit set of $x$ under $T$ is the whole space, i.e. $J_T(x)=X$ (the extended mixing limit set of $x$ under
$T$ is the whole space, i.e. $J_T^{mix}(x)=X$). In this case $x$ will be called a \textit{$J$-class vector}(\textit{$J^{mix}$-class vector}) for $T$ and the set of all
$y\in X$ with $J_T(y)=X$ ($J_T^{mix}(y)=X$) is denoted by $A_T$ ($A^{mix}_T$).
\end{definition}

It is very easy to see that $T\in L(X)$ is topologically transitive if and only if $J_T(x)=X$ for every $x\in X$ and that $T$ is topologically mixing if and only if
$J_T^{mix}(x)=X$ for every $x\in X$. This observation justifies the term ``local topological transitive". For examples of operators which are locally topologically
transitive (locally topologically mixing) but not topologically transitive (topologically mixing) as well as properties of $J$-class operators we refer to \cite{cosma1},
\cite{CosMa}. For recent additional work on local topological transitivity see \cite{AzMu}, \cite{Na}.\smallskip

By $l^p(\mathbb{N})$ $1\leq p<+\infty$ we denote the space of $p$-summable sequences of complex numbers which becomes a Banach space under the usual $p$-norm. Recall
that, given a sequence of positive and bounded weights $w=(w_n)$, the operator $T_w\in L(l^p(\mathbb{N})$, $1\leq p<+\infty$ defined by $T_w(x_1,x_2,\ldots )= (w_1x_2,
w_2x_3,\ldots )$ is called a backward unilateral weighted shift. A beautiful result of Salas \cite{Salas} asserts that if $T_w: l^p(\mathbb{N})\to l^p(\mathbb{N})$ is a
backward unilateral weighted shift then $I+T_w$ is hypercyclic on $l^p(\mathbb{N})$, $1\leq p<+\infty$; actually it is even mixing \cite{Gri}. Therefore, it is natural
to seek a ``localized analogue" of Salas result, if it exists, in the setting of weighted shifts acting on $l^{\infty}(\mathbb{N})$. Of course, when we say ``localized
analogue" we mean that, one should replace the term ``hypercyclicity" by ``local topological transitivity", since $l^{\infty}(\mathbb{N})$ is non-separable. So, in view
of the above, the following question arises naturally.

\medskip

\noindent {\bf Question}. \textit{Let $T_w:l^{\infty}(\mathbb{N})\to l^{\infty}(\mathbb{N})$ be a backward unilateral wei\-ghted shift. Is it true that $I+T_w$ is
locally topologically transitive on $l^{\infty}(\mathbb{N})$}?
\medskip

The purpose of the present note is to give a negative answer to this question by proving the following

\begin{theorem}\label{main}
Let $B:l^{\infty}(\mathbb{N})\to l^{\infty}(\mathbb{N})$ be the unweighted backward unilateral shift, i.e.
$$B(x_1,x_2,\ldots ):= (x_2, x_3,\ldots ), \,\,\, \textrm{for}\,\,\, x=(x_1,x_2,\ldots )\in l^{\infty}(\mathbb{N}).$$
\begin{enumerate}
\item[(i)] If $| \lambda| \leq 2$ the set $J_{I+\lambda B}(x)$ has empty interior for every $x\in X$ and, in
particular, the operator $I+\lambda B$ is not locally topologically transitive.

\item[(ii)] If $| \lambda| >2$ then $I+\lambda B$ is locally topologically mixing and $A^{mix}_{I+\lambda B}=
c_0(\mathbb{N})$.
\end{enumerate}
\end{theorem}
Part (i) of Theorem \ref{main} gives a negative answer to the above question. Actually, Theorem \ref{main}
provides a characterization of $J^{mix}$-class operators of the form $I+\lambda B$, $\lambda \in \mathbb{C}$ and
in addition describes the set $A^{mix}_{I+\lambda B}$. Our paper is organized as follows. In section 2 we give
the proof of Theorem \ref{main}. In the final section, Section 3, we establish a variant of Theorem \ref{main}
and we also exhibit some information on the structure of the set $A_T^{mix}$ for certain $T\in
L(l^{\infty}(\mathbb{N}))$.

\section{Proof of the main result}

\subsection{Preparatory lemmas}
We start with a lemma which can be found in \cite[Corollary 3.4]{cosma1}.

\begin{lemma}\label{denserange}
Let $T:X\to X$ be an operator. Suppose there exists a vector $x\in X$ such that the set $J_T(x)$ has non-empty
interior. Then for every $\lambda\in\mathbb{C}$ with $|\lambda|\leq 1$ the operator $T-\lambda I$  has dense
range.
\end{lemma}

The next lemma also appears in \cite{cosma1} and gives information on the spectrum of a $J$-class operator.
\begin{lemma} \label{unitcirclespectrum}
Let $X$ be a Banach space and $T\in L(X)$. If $T$ is $J$-class then the spectrum of $T$, $\sigma (T)$,
intersects the unit circle $\partial\mathbb{D}$.
\end{lemma}

\begin{lemma} \label{tml1}
Let $T$ be an operator on a Banach space $X$. If $0\in J_{T}^{mix}(x)$ for some $x\in X$ then
$J_{T}^{mix}(x)=J_{T}^{mix}(0)$.
\end{lemma}
\begin{proof}
Assume that $0\in J_{T}^{mix}(x)$. Then there exists a sequence $(x_n)\subset X$ such that $x_{n}\rightarrow x$
and $T^nx_{n}\rightarrow 0$. Let us show first that $J_{T}^{mix}(0)\subset J_{T}^{mix}(x)$. Indeed, take $y\in
J_{T}^{mix}(0)$. There exists a sequence $(y_n)\subset X$ such that $y_{n}\rightarrow 0$ and
$T^ny_{n}\rightarrow y$. Hence $T^n(x_n+y_n)\to y$ and $x_n+y_n\to x$. Therefore $y\in J_{T}^{mix}(x)$. For the
converse inclusion let $y\in J_{T}^{mix}(x)$. There exists a sequence $(v_n)\subset X$ such that
$v_{n}\rightarrow x$ and $T^nv_{n}\rightarrow y$. Hence $T^n(v_n-x_n)\to y$, $v_n-x_n\to 0$ and we conclude that
$y\in J_{T}^{mix}(0)$. This completes the proof.
\end{proof}

\begin{lemma} \label{kernel}
Let $T$ be an operator on a Banach space $X$. If $J_T^{mix}(0)=X$ then $J_T^{mix}(x)=X$ for every $x\in Ker T$.
\end{lemma}
\begin{proof}
Observe that $0\in J_T^{mix}(x)$ for every $x\in Ker T$ and we conclude by Lemma \ref{tml1}.
\end{proof}

In the next elementary, but very useful in what follows, lemma we calculate the primage of a vector $y=(y_n)\in
\mathbb{C}^{\mathbb{N}}$ under $I+\lambda B$, where $\lambda\in\mathbb{C}\setminus \{ 0\}$.

\begin{lemma}\label{preimage}
Let $B:\mathbb{C}^{\mathbb{N}}\to \mathbb{C}^{\mathbb{N}}$ be the unweighted backward unilateral shift and let
$x=(x_n)$, $y=(y_n)$ be vectors in $\mathbb{C}^{\mathbb{N}}$ such that $(I+\lambda B)x=y$ for some $\lambda \in
\mathbb{C}\setminus \{ 0\}$. Then,
\[
x_n=\frac{-1}{(-\lambda)^n} \sum_{k=1}^{n-1} (-\lambda)^k y_k + \frac{x_1}{(-\lambda)^{n-1}}
\]
for every $n=2,3,\ldots $.
\end{lemma}
\begin{proof}
Notice that $x_n+\lambda x_{n+1}=y_n$, for every $n\in\mathbb{N}$. It is easy to show that the kernel of
$I+\lambda B$ consists of all vectors $w=(w_n)$ such that $w_n=\frac{w_1}{(-\lambda)^{n-1}}$, for every
$n\in\mathbb{N}$. Define the vector $z=(z_n)$ by $z_n:=\frac{-1}{(-\lambda)^n} \sum_{k=1}^{n-1} (-\lambda)^k
y_k$ for $n=2,3,\ldots $ and $z_1:=0$. Straightforward calculations give $(I+\lambda B)z=y$. Thus, $(I+\lambda
B)(x-z)=0$ and from the above description of the kernel of $(I+\lambda B)$ we conclude that
$x_n=\frac{-1}{(-\lambda)^n} \sum_{k=1}^{n-1} (-\lambda)^k y_k + \frac{x_1}{(-\lambda)^{n-1}}$ for every
$n=2,3,\ldots $.
\end{proof}

\begin{lemma} \label{lambda1}
Let $\lambda \in \mathbb{C}$ with $|\lambda |=1$. Then the operator $I+\lambda B:l^{\infty}(\mathbb{N})\to
l^{\infty}(\mathbb{N})$ does not have dense range.
\end{lemma}
\begin{proof}
Consider the vector $y=(y_n)\in l^{\infty}(\mathbb{N})$ with $y_n=(-\lambda)^{-n}$,  for every $n\in\mathbb{N}$.
We will show that the open ball $B(y,1/2)$ centered at $y$ with radius $\frac{1}{2}$ does not intersect the
range of $I+\lambda B$. Assume the contrary, i.e. there is a vector $w=(w_n)\in B(y,1/2)$ and a vector $x\in
l^{\infty}(\mathbb{N})$ such that $(I+\lambda B)x=w$. By Lemma \ref{preimage} we deduce that

\begin{equation} \label{eq2}
x_n=\frac{-1}{(-\lambda)^n} \sum_{k=1}^{n-1} (-\lambda)^k w_k + \frac{x_1}{(-\lambda)^{n-1}},
\end{equation}
for every $n=2,3,\ldots $. Since $w\in B(y,1/2)$ we have $w_k=y_k+\varepsilon_k$ for some $\varepsilon_k \in
\mathbb{C}$ with $|\varepsilon_k |< \frac{1}{2}$, for every $k\in\mathbb{N}$. Substituting $w_k$ in (\ref{eq2})
we get

\[
\left | \sum_{k=1}^{n-1} (-\lambda)^k (y_k+\varepsilon_k) - \lambda x_1  \right | \leq  \| x\|_{\infty}
\]
for every $n\in\mathbb{N}$, which in turn implies
\[
\left | n-1 + \sum_{k=1}^{n-1} (-\lambda)^k \varepsilon_k   \right | \leq \| x\|_{\infty} + |x_1|\leq 2\, \|
x\|_{\infty},
\]
since $y_k=(-\lambda)^{-k}$. Now, the estimate $|\sum_{k=1}^{n-1} (-\lambda)^k \varepsilon_k | \leq
\frac{n-1}{2}$ and the triangle inequality yield
\[
\frac{n-1}{2}=n-1 -\frac{(n-1)}{2} \leq \left | n-1 + \sum_{k=1}^{n-1} (-\lambda)^k \varepsilon_k \right | \leq
2\, \| x\|_{\infty}
\]
for every $n\in\mathbb{N}$, which is a contradiction.
\end{proof}

For the next two lemmas we need to introduce some terminology and notation. Let $X$ be a Banach space and let
$T:X\to X$ be a bounded linear operator. For a closed and $T$-invariant subspace $M$ of $X$ we define the
induced operator
$$\widehat{T}:X/M\to X/M \,\,\, \textrm{by}\,\,\, \widehat{T}[x]_M:=[Tx]_M.$$ $\hat{T}$ is well defined, linear and
continuous in the induced quotient topology. In addition we have $\| \widehat{T}\| \leq \| T\|$. In the
following we will just write $[x]$ instead of $[x]_M$.

\begin{lemma} \label{inducedJmix}
Let $X$ be a Banach space and $T\in L(X)$. Assume that $T$ is $J^{mix}$-class and let $M$ be a closed and
$T$-invariant subspace of $X$ such that $M\subset A_T^{mix}$ and $A_T^{mix}\setminus M\neq \emptyset$. Then the
induced operator $\widehat{T}:X/M\to X/M$ is $J^{mix}$-class.
\end{lemma}
\begin{proof}
Let $x\in A_T^{mix}\setminus M$. Then $[x]\neq 0$. Consider any $[y]\in X/M$. There exists a sequence $(x_n)$ in
$X$ such that $x_n\to x$ and $T^nx_n\to y$. From the last we conclude that $[x_n]\to [x]$ and
$\widehat{T}^n[x_n]=[T^nx_n]\to [y]$. This shows that $\widehat{T}$ is $J^{mix}$-class.
\end{proof}

\begin{lemma} \label{spectrumquotient}
Let $B$ be the backward shift on $l^{\infty}(\mathbb{N})$ and consider the induced operator
$\hat{B}:l^{\infty}(\mathbb{N})/c_0(\mathbb{N}) \to l^{\infty}(\mathbb{N})/c_0(\mathbb{N})$. Then $\sigma
(\hat{B})=\partial{\mathbb{D}}$.
\end{lemma}
\begin{proof}
Choose any $\mu:=e^{i\theta} \in \partial{\mathbb{D}}$, $\theta \in \mathbb{R}$. It is straightforward to check
that the vector $[(x_n)] \in l^{\infty}(\mathbb{N})/c_0(\mathbb{N})$ with $x_{n}=e^{i\theta (n-1)}$ for
$n=1,2,\ldots$ is an eigenvector for $\hat{B}$ corresponding to $\mu$. Hence, $\partial{\mathbb{D}} \subset
\sigma (\hat{B})$. It remains to show the converse inclusion. To this end, let $\mu \in \mathbb{C}$ with $|\mu
|<1$. Then $\hat{B}-\mu I$ is injective. Indeed, if not, there exists $x=(x_n)\notin c_0(\mathbb{N})$ such that
$(\hat{B}-\mu I)[x]=0$ and therefore $(x_{n+1}-\mu x_n)\in c_0(\mathbb{N})$. Set $\delta:=\limsup_n|x_n|$ and
let $(n_k)$ be a strictly increasing sequence of positive integers such that $| x_{n_k+1}|\to \delta$. Defining
$\epsilon_n:=x_{n+1}-\mu x_n$, then $| x_{n_k+1}|\leq |\mu| |x_{n_k}|+|\epsilon_{n_k}|$ and we deduce that
$$\delta=\lim_{k\to+\infty}|x_{n_k+1}|\leq |\mu| \limsup_k|x_{n_k}|.$$ Since $|\mu |<1$ and $\delta >0$, the
above inequality gives a contradiction. It is now easy to show that $\hat{B}-\mu I$ is surjective. Let $[y]\in l^{\infty}(\mathbb{N})/c_0(\mathbb{N})$ with $y=(y_n)$.
Then we have $(\hat{B}-\mu I)[x]=[y]$, where $x=(x_n)$ belongs to $l^{\infty}(\mathbb{N})$ and it is defined by $x_1=0$ and $x_{n+1}:= \sum_{k=0}^{n-1}y_{n-k}\mu^k$,
$n=1,2,\ldots $. Therefore $\hat{B}-\mu I$ is invertible and since $\| \hat{B}\| \leq \| B\| =1$ the conclusion follows.
\end{proof}

\subsection{Proof of Theorem \ref{main}.}
If $\lambda =0$ then, clearly, the $J$-sets of the identity operator are singletons. So, from now on, we may
assume that $\lambda\neq 0$.

\medskip

\noindent (i) Let $0<| \lambda| \leq 2$. We can rewrite the previous inequality as $||\lambda|-1|\leq 1$. Assume
that $J(x)$ has non-empty interior for some $x\in X$. Then, by Lemma \ref{denserange}, the operator $(I+\lambda
B) -(1-|\lambda|)I$ has dense range. Writing $(I+\lambda B) -(1-|\lambda|)I =|\lambda|(I+\frac{\lambda}{|\lambda
|}B)$ we conclude that the operator $I+\frac{\lambda}{|\lambda |}B$ has dense range, which contradicts Lemma
\ref{lambda1}.

\medskip

\noindent (ii) Let $\lambda\in\mathbb{C}$ with $|\lambda|>2$. Fix a vector $y=(y_n)\in l^{\infty}(\mathbb{N})$
and let $x=(x_n)\in \mathbb{C}^{\mathbb{N}}$ be such that $(I+\lambda B)x=y$. Then, by Lemma \ref{preimage},
\[
x_n=\frac{-1}{(-\lambda)^n} \sum_{k=1}^{n-1} (-\lambda)^k y_k + \frac{x_1}{(-\lambda)^{n-1}},
\]
for every $n=2,3,\ldots $. If we take $x_1=0$ then it readily follows that $| x_n| \leq \frac{\|
y\|_{\infty}}{|\lambda| -1}$ for every $n=2,3,\ldots$. Therefore, defining $w^{(1)}=(w^{(1)}_n)\in
\mathbb{C}^{\mathbb{N}}$ by $$w^{(1)}_1:=0,\,\,\, w^{(1)}_n:=\frac{-1}{(-\lambda)^n} \sum_{k=1}^{n-1}
(-\lambda)^k y_k + \frac{x_1}{(-\lambda)^{n-1}} \,\,\, n=2,3,\ldots$$ we have $\| w^{(1)}\|_{\infty}\leq
\frac{\| y\|_{\infty}}{|\lambda| -1}$, hence $w^{(1)}\in l^{\infty}(\mathbb{N})$, and $(I+\lambda B)w^{(1)}=y$.
If we repeat the same argument for $w^{(1)}$ in place of $y$ we find a vector $w^{(2)}\in
l^{\infty}(\mathbb{N})$ such that $(I+\lambda B)w^{(2)}=w^{(1)}$ and $\| w^{(2)}\|_{\infty}\leq \frac{\|
w^{(1)}\|_{\infty}}{|\lambda| -1}\leq \frac{\| y\|_{\infty}}{(|\lambda| -1)^2}$. Proceeding inductively, for
every positive integer $n$ we find a vector $w^{(n)}\in l^{\infty}(\mathbb{N})$ such that $(I+\lambda
B)^nw^{(n)}=y$ and $\| w^{(n)}\|_{\infty}\leq \frac{\| y\|_{\infty}}{(|\lambda| -1)^n}$. Since $|\lambda |>2$,
$w^{(n)}\to 0$. Thus, $J_{I+\lambda B}^{mix}(0)=l^{\infty}(\mathbb{N})$. Noticing that the kernel of $I+\lambda
B$ in $l^{\infty}(\mathbb{N})$ is non-trivial and taking any non-zero vector $w$ in this kernel, it follows that
$J_{I+\lambda B}^{mix}(w)=l^{\infty}(\mathbb{N})$ by Lemma \ref{kernel}. Thus, $I+\lambda B$ is locally
topologically mixing. It remains to show that $A_{I+\lambda B}^{mix}=c_0(\mathbb{N})$. Take any $x\in
c_0(\mathbb{N})$. The restricted operator $(I+\lambda B)|_{c_0(\mathbb{N})}:c_0(\mathbb{N})\to c_0(\mathbb{N})$
is topologically mixing, see \cite{Gri}. This, implies that $c_0(\mathbb{N}) \subset J_{I+\lambda B}^{mix}(x)$.
In particular, $0\in J_{I+\lambda B}^{mix}(x)$ and by Lemma \ref{tml1} we get $J_{I+\lambda
B}^{mix}(x)=J_{I+\lambda B}^{mix}(0)$. Since $J_{I+\lambda B}^{mix}(0)=l^{\infty}(\mathbb{N})$ we conclude that
$c_0(\mathbb{N}) \subset A_{I+\lambda B}^{mix}$. To show the converse inclusion we argue by contradiction, so
assume that $A_{I+\lambda B}^{mix}\setminus  c_0(\mathbb{N}) \neq \emptyset$. Now, Lemma \ref{inducedJmix}
implies that $\widehat{I+\lambda B}$ is $J^{mix}$-class and in view of Lemma \ref{unitcirclespectrum} the
spectrum of $\widehat{I+\lambda B}$ should intersect the unit circle $\partial \mathbb{D}$. However, by Lemma
\ref{spectrumquotient} and the spectral mapping theorem it follows that $\sigma ( \widehat{I+\lambda B} )=\{
1+\lambda e^{i\theta }: \theta \in \mathbb{R} \}$. Hence, $\sigma ( \widehat{I+\lambda B} )$ does not intersect
the unit circle since $|\lambda |>2$, which is a contradiction. This finishes the proof of the theorem. \qed

\section{Further results}

\subsection{A variation of Theorem \ref{main}}
The proof of the next lemma is similar to the proof of item (i) of Proposition 5.9 in \cite{cosma1} and it is
left to the interested reader.
\begin{lemma} \label{lemma1fr}
Let $T:X\rightarrow X$ be an operator acting on a Banach space $X$. Then $J_{T}^{mix}(0)=J_{T^n}^{mix}(0)$ for
every positive integer $n$.
\end{lemma}

For $r>0$ the symbol $D(0,r)$ stands for the open disk in the complex plane with center $0$ and radius $r$.
\begin{theorem} \label{thm2}
Let $f$ be a holomorphic function on $D(0,r)$ for $r>1$ such that
\[\overline{\mathbb{D}}\subseteq f(\overline{\mathbb{D}}) .\]
Then there exists a positive number $R_0$ such that for all $R>R_0$ the operator $Rf(B):
l^{\infty}(\mathbb{N})\to l^{\infty}(\mathbb{N})$ is $J^{mix}$-class.
\end{theorem}
\begin{proof} Since $f$ is holomorphic in $D(0,r)$ with $r>1$, $f$ has a finite number of zeros on $\overline{\mathbb{D}}$,
say $z_k$, $\ k=1,\ldots ,n $ for some $n\in \mathbb{N}$. The assumption $\overline{\mathbb{D}}\subseteq
f(\overline{\mathbb{D}})$ implies that $\emptyset=\mathbb{D}\cap\partial f(\overline{\mathbb{D}})=\mathbb{D}\cap
f(\partial\mathbb{D})$ from which it follows that $|z_k|<1$ for every $\ k=1,\ldots ,n $. Then there exists a
holomorphic function $g$ on $D(0,r')$ for some $r'\in (1,r]$ such that $f(z)=g(z)(z-z_1)\cdot\ldots\cdot(z-z_n)$
and $g(z)\neq 0$ for all $z\in D(0,r')$. It follows that $g(\overline{\mathbb{D}})\subseteq\mathbb{C}\backslash
D(0,\delta )$ for some $\delta>0$ and hence we can choose an $R_1>1$ such that
$R_1g(\overline{\mathbb{D}})\subset\mathbb{C}\backslash\overline{\mathbb{D}}$ and
$R_1\prod_{k=1}^{n}||z_k|-1|>1$. Define now  \[h(z):=R_1^2f(z)=\tilde{h}(z)\cdot p(z),\] where
$p(z):=R_1(z-z_1)\cdot\ldots\cdot(z-z_n)$ and $\tilde{h}(z)=R_1g(z)$.\smallskip

\noindent {\bf Claim}. Let $y\in l^{\infty}(\mathbb{N})$. There exists $x\in l^{\infty}(\mathbb{N})$ such that: \[p(B)x=y \,\,\, \textrm{and} \,\,\, \| x\| \leq
\frac{1}{R_1\prod_{k=1}^{n}||z_k|-1|} \| y\| .\]

\noindent\textit{Proof of Claim}. We have $p(B)=R_1(B-z_1I)\cdots (B-z_nI)$. Suppose that $z_j\neq 0$ for every $j=1,\ldots ,n$. Now we can write
$B-z_nI=-z_n(\frac{1}{(-z_n)}B+I)$ and then applying Lemma \ref{preimage} and following the proof of item (ii) in Theorem \ref{main} we get
$$ \left(\frac{1}{(-z_n)}B+I\right)x^{(n)}=\frac{1}{R_1(-z_n)}y \,\,\, \textrm{and} \,\,\,\| x^{(n)} \|\leq \frac{\| y\| }{
R_1|z_n| \left(\frac{1}{|z_n|} -1\right)} $$ for some $x^{(n)}\in l^{\infty}(\mathbb{N})$, or equivalently
$$ R_1(B-z_nI)x^{(n)}=y \,\,\, \textrm{and} \,\,\,\| x^{(n)} \|\leq \frac{\| y\|}{R_1(1-|z_n|)} .$$
Arguing as before, there exists $x^{(j-1)}\in l^{\infty}(\mathbb{N})$ such that
$$(B-z_{j-1}I)x^{(j-1)}=x^{(j)}\,\,\, \textrm{and} \,\,\,\| x^{(j-1)} \|\leq \frac{\| x^{(j)}\|}{1-|z_{j-1}|} ,$$
for every $j=n, n-1, \ldots ,2$. Setting $x:=x^{(1)}$, we conclude that
$$ p(B)x=y \,\,\, \textrm{and} \,\,\,\| x\|
\leq \frac{1}{R_1\prod_{k=1}^{n}||z_k|-1|} \| y\| .$$ Of course if $z_j=0$ for some $j\in \{ 1,\ldots ,n\}$ a
similar argument can be applied without any difficulty. This completes the proof of the claim.\smallskip

By the spectral theorem we get
\[\sigma(\tilde{h}(B))=\tilde{h}(\sigma(B))=R_1 g(\sigma(B))=
R_1 g(\overline{\mathbb{D}})\subset \mathbb{C}\backslash{\overline{\mathbb{D}}}.\] It follows that the inverse operator $\tilde{h}(B)^{-1}$ exists and
$\sigma(\tilde{h}(B)^{-1})\subset \mathbb{D}$. Therefore, for a given $0<a<1$ there exists $n_0\in\mathbb{N}$ such that
\[\left\|\tilde{h}(B)^{-n_0}\right\|\leq  a.\]
Consider now any $y\in l^\infty$. The equation $\tilde{h}(B)^{n_0}w^{(1)}=y$ has a unique solution $w^{(1)}\in l^{\infty}(\mathbb{N})$ and
\[\left\|w^{(1)}\right\|=\left\|\tilde{h}(B)^{-n_0}y\right\|\leq a\left\|y\right\|.\]
Define $$b:=\frac{1}{\left(R_1\prod_{k=1}^{n}||z_k|-1|\right)^{n_0}}<1.$$ Applying the Claim $n_0$ times, there
exists a vector $v^{(1)}\in l^{\infty}(\mathbb{N})$ such that $\left\|v^{(1)}\right\|\leq
b\left\|w^{(1)}\right\|$ and $p(B)^{n_0}v^{(1)}=w^{(1)}$. Altogether we get
\[h(B)^{n_0}v^{(1)}=y \text{ with } \left\|v^{(1)}\right\|\leq b\left\|w^{(1)}\right\|\leq ab\left\|y\right\|.\]
The equation $\tilde{h}(B)^{n_0}w^{(2)}=v^{(1)}$ has a unique solution $w^{(2)}\in l^{\infty}(\mathbb{N})$ and
\[\left\|w^{(2)}\right\|=\left\|\tilde{h}(B)^{-n_0}v^{(1)}\right\|\leq a\left\|v^{(1)}\right\|.\]
Again by the Claim there exists $v^{(2)}\in l^{\infty}(\mathbb{N})$ such that $\left\|v^{(2)}\right\|\leq
b\left\|w^{(2)}\right\|$ and $p(B)^{n_0}v^{(2)}=w^{(2)}$. From the above we have
\[h(B)^{2n_0}v^{(2)}=y \text{ and } \left\|v^{(2)}\right\|\leq b\left\|w^{(2)}\right\|\leq (ab)^2\left\|y\right\|.\]
Proceeding inductively, we construct a sequence $(v^{(m)})$ in $l^{\infty}(\mathbb{N})$ such that
\[h(B)^{mn_0}v^{(m)}=y \text{ and } \left\|v^{(m)}\right\|\leq (ab)^m\left\|y\right\|\]
for every positive integer $m$. Since $ab<1$, $v^{(m)}\to 0$ and we conclude that $J_{h(B)^{n_0}}^{mix}(0)=l^{\infty}(\mathbb{N})$. Now, in view of Lemma \ref{lemma1fr},
$J_{h(B)}^{mix}(0)=l^{\infty}(\mathbb{N})$. Observe that $p(B)$ has non-trivial kernel, which in turn implies that $h(B)$ has non-trivial kernel. Hence, for every
non-zero vector $x\in Ker(h(B))$, $J_{h(B)}^{mix}(x)=l^{\infty}(\mathbb{N})$. Clearly, for every $R\geq R_0$, where $R_0:=R_1^2$, the operator $Rf(B)$ is locally
topologically mixing in $l^{\infty}(\mathbb{N})$. This completes the proof of the theorem.
\end{proof}

\begin{corollary}
Let $f$ be a holomorphic function on $D(0,r)$ for some $r>1$ such that $0$ belongs to the interior of $f(\overline{\mathbb{D}})$. Then there exists $R_0>0$ such that $R
f(B):l^{\infty}(\mathbb{N})\to l^{\infty}(\mathbb{N})$ is $J^{mix}$-class for every $R>R_0$.
\end{corollary}

\subsection{Properties of the set $A^{mix}_T$}

Our last result in this section concerns the structure of the set $A^{mix}_T$ of a $J^{mix}$-class operator on
$l^{\infty}(\mathbb{N})$. Observe that if $T$ is $J^{mix}$-class on a Banach space $X$ then $A^{mix}_T$ is a
closed subspace. So far, in all examples of $J^{mix}$-class operators $T$ on $l^{\infty}(\mathbb{N})$, the
closed subspace $A^{mix}_T$ is separable. For instance, if $\lambda\in \mathbb{C}$ with $|\lambda |>1$ then
$\lambda B: l^{\infty}(\mathbb{N})\to l^{\infty}(\mathbb{N})$ is $J^{mix}$-class and $A^{mix}_{\lambda
B}=c_0(\mathbb{N})$, see \cite{cosma1}. A similar result holds for the operator $I+\lambda
B:l^{\infty}(\mathbb{N})\to l^{\infty}(\mathbb{N})$ whenever $|\lambda |>2$, as Theorem \ref{main} shows. In the
following proposition we exhibit an operator $T$ on $l^{\infty}(\mathbb{N})$ such that $A^{mix}_T$ is
non-separable.

\begin{proposition}
There exists a $J^{mix}$-class operator $T$ on $l^{\infty}(\mathbb{N})$ such that $A^{mix}_T$ is non-separable.
\end{proposition}
\begin{proof}
Fix an isomorphism $S:l^{\infty}(\mathbb{N})\rightarrow l^{\infty}(\mathbb{N})\oplus l^{\infty}(\mathbb{N})$ and consider the projection $P:l^{\infty}(\mathbb{N})\oplus
l^{\infty}(\mathbb{N})\rightarrow l^{\infty} (\mathbb{N})$ defined by $P(x\oplus y):=x$. Now choose $|\lambda|>\left\|S^{-1}\right\|$ and define the operator
$T:l^\infty(\mathbb{N})\rightarrow l^\infty(\mathbb{N})$ by $T:=\lambda B\circ P\circ S$. Then $T$ is $J^{mix}$-class with $M:=S^{-1}(\{0\}\oplus
l^\infty(\mathbb{N}))\subset A_{T}^{mix}$. To see this take any $y=(y_n)\in l^{\infty}(\mathbb{N})$. Define $x^{(1)}:=(0,\lambda^{-1}y_1,\lambda^{-1}y_2,\ldots)$. Then
$\lambda Bx^{(1)}=y$ and $\left\|x^{(1)}\right\|=\frac{1}{|\lambda|}\left\|y\right\|$. Note that $P(x^{(1)}\oplus 0)=x^{(1)}$ with $\left\|x^{(1)}\oplus
0\right\|=\left\|x^{(1)}\right\|$. Choose now $z^{(1)}\in l^{\infty}(\mathbb{N})$ with $Sz^{(1)}=x^{(1)}\oplus 0$.  Then $Tz^{(1)}=y$ and
\[\left\|z^{(1)}\right\|\leq\left\|S^{-1}\right\|\left\|x^{(1)}\oplus 0\right\|=\left\|S^{-1}\right\|\cdot\left\|x^{(1)}
\right\|=\frac{\left\|S^{-1}\right\|}{|\lambda|}\left\|y\right\|=q\left\|y\right\| ,\] where $q:=\left\|S^{-1}\right\||\lambda|^{-1}<1$. Proceeding in this way, we find
$z^{(2)}\in l^{\infty}(\mathbb{N})$ with $Tz^{(2)}=z^{(1)}$ and $\left\|z^{(2)}\right\|\leq q\left\|z^{(1)}\right\|\leq q^2\left\|y\right\|$. Inductively, we get a
sequence $(z^{(n)})$ in $l^{\infty}(\mathbb{N})$ such that $T^{n}z^{(n)}=y$ and $z^{(n)}\rightarrow 0$. Since $y$ is arbitrary we conclude that
$J_{T}^{mix}(0)=l^{\infty}(\mathbb{N})$. It is easy to see that $M:=S^{-1}(\{0\}\oplus l^{\infty}(\mathbb{N}))$ is non separable and $M \subset KerT$, where $KerT$ is
the kernel of $T$. Since $J_{T}^{mix}(0)=l^{\infty}(\mathbb{N})$ we have $KerT\subset A_{T}^{mix}$, and the conclusion follows.
\end{proof}

As we observed above for any $T\in L(X)$ the set $A_T^{mix}$ is a closed subspace. However  for the set $A_T$
the situation is less clear and so we ask the following

\smallskip

\noindent {\bf Question}. Does there exist a $J$-class operator $T$ acting on some Banach space $X$ such that $A_T$ is not a vector space?

\end{document}